\documentclass[12pt]{amsart}
\usepackage{mathrsfs}
\usepackage{amsmath}
\usepackage{verbatim}
\usepackage{hyperref}
\usepackage{amsfonts}
\usepackage{bbm}
\usepackage{young}
\usepackage{amscd}
\usepackage{color}
\usepackage{amssymb}
\usepackage{graphicx}
\usepackage[all, knot]{xy}
\usepackage[x-cp1252]{inputenx}
\usepackage{CJK}
\usepackage{cite}
\xyoption{arc} 

\usepackage{datetime}
\usepackage{amstext}
\usepackage{mathtools}
\usepackage{amsxtra}
\usepackage{amsopn}
\usepackage{amsthm}
\newtheorem{thm}{Theorem}[section]
\theoremstyle{definition}

\theoremstyle{plain}
\newtheorem*{mr}{Main Theorem}
\theoremstyle{remark}

\theoremstyle{plain}

\theoremstyle{plain}

\theoremstyle{plain}
\newtheorem{prop}[thm]{Proposition}
\theoremstyle{conjecture}

\usepackage{cancel}

\newcommand{\sym}[1][]{\ensuremath{{\rm Sym}^{#1}}}              
\newcommand{\Sym}{{\rm Sym}}

\newcommand{\mc}{\mathcal}
\CompileMatrices
\usepackage{listings}
\title{A Sheaf Cohomology Restriction Formula on toric complete intersections}
\author{Zhentao Lyu}
\address{School of Mathematics and Statistics, Beijing Institute of Technology, Beijing 100081, P.R. China}
\email{zhentao@sas.upenn.edu}
\date{\today}
\begin{document}
\maketitle
\begin{abstract}
We prove a sheaf cohomology restriction (SCORE) formula for a class of vector bundles on complete intersections in toric varieties. The formula enables one to compute cohomology products on the complete intersection $X$ via computations on the ambient space $V$ and potentially compute certain quantum corrections to the classical sheaf cohomology ring. Schematically the formula reads $(s_1,...,s_r)_X = (s_1,...,s_r, g)_V$ with $g$ being an explicitly described quantity derived from the monad data of the bundle.
\end{abstract}
\section{Introduction}

For a compact K\"ahler manifold $X$ and a holomorphic vector bundle $\mc{E}$, 
we are interested in studying the \textit{polymology} $$H_\mc{E}^*(X):=\bigoplus_{p,q} H^q(X, \wedge^p \mc{E}^*),$$ where $\mc{E}^*$ is the dual of $\mc{E}$. This is an associative algebra \cite{donagi2014mathematical}. In the special case when $\mc{E}$ is the tangent bundle $T_X$, it is the familiar cohomology of differential forms. The polymology, or classical sheaf cohomology ring, has been studied on toric varieties \cite{donagi2014mathematical} and Grassmannians \cite{guo2017classical}, particularly motivated by the efforts to mathematically understand the physics theory of A/2-twisted gauge linear sigma models introduced by Witten \cite{witten1993phases}. See also
\cite{katz2006notes,mcorist2008half,mcorist2009summing,kreuzer20110,mcorist2011revival, donagi2013physical,Guo2017,melnikov2019introduction,guo2020quantum} for some physics developments. The mathematical endeavor is to define quantum sheaf cohomology $QH^*_\mc{E}(X)$ for $\mc{E}$ with special properties, which is a ``quantum correction" of $H^*_\mc{E}(X)$. Mathematically this has only been done on toric varieties \cite{donagi2014mathematical}, while \cite{mcorist2009summing} works out a quantum restriction formula for toric complete intersections via localization techniques in physics. They motivate our study of  $H^*_\mc{E}(X)$ on hypersurfaces and complete intersections in toric varieties as well as its quantum corrections.

Let $X$ be a smooth hypersurface in a smooth projective variety $V$. Consider vector bundles $\mc{E}^*_X$ and $\mc{E}^*_V$ on $X$ and $V$ respectively, with the requirement that they fit in the short exact sequence (SES)
\begin{equation}
0 \to \mc{O}_X(-X) \xrightarrow{J}  \mc{E}_V^*|_X \to \mc{E}^*_X \to 0.
\end{equation}
We would like to compute the cohomology product $$c_X: H^1(\mc{E}^*_X) \times ... \times H^1(\mc{E}^*_X) \to H^{n-1}(\wedge^{n-1}\mc{E}^*_X)$$
and relate it to 
$$c_V: H^1(\mc{E}^*_V) \times ... \times H^1(\mc{E}^*_V) \to H^n(\wedge^n\mc{E}^*_V).$$ Let $\sigma_i \in H^1(\mc{E}^*_V), i =1,...,n$, denote the image $c_V(\sigma_1,...,\sigma_n)$ by $\langle \sigma_1,...,\sigma_n \rangle_V$.
If we have a restriction map \begin{equation*}
\begin{array}{rll}
H^1(\mc{E}^*_V) & \to & H^1(\mc{E}^*_X)
\\
\sigma &\mapsto &\bar{\sigma}
\end{array},
\end{equation*}
we would like to have a sheaf cohomology restriction (SCORE) formula in the form
\begin{equation}
\langle \bar{\sigma}_1,...,\bar{\sigma}_{n-1} \rangle_X = \langle \sigma_1,...,\sigma_{n-1}, [\mc{E}^*_X] \rangle_V,
\end{equation}
with some class $[\mc{E}^*_X]\in H^1(\mc{E}^*_V)$.

As an example, note that in the case when $\mc{E}^*_V$ is the cotangent bundle $\Omega_V$,  let $D_1,...,D_r$ be divisors of $V$, and they restrict to divisors on $X\subset V$. 
Then the intersection of $D_1|_X,...,D_r|_X$ can be computed by 
\begin{equation}
( [D_1|_X],...,[D_r|_X])_X = ( [D_1],...,[D_r] ,[X] )_V. 
\end{equation}

In this paper, we prove a SCORE formula for the case where $V$ is an $n$ dimensional smooth projective toric variety and $\mc{E}^*_V$ is a small `toric' deformation of the cotangent bundle, and $X$ is a smooth complete intersection, of codimension $m$. We make use of a deformed version of the toric Euler sequence of the cotangent bundle, so that both $\mc{E}^*_V$ and $\mc{E}^*_X$ can be expressed in monads with terms splitting into direct sums of line bundles. This enables us to make sense of the restriction map, the comparison between $H^{n-m}(\wedge^{n-m}\mc{E}^*_X)$ and $H^n(\wedge^n\mc{E}^*_V)$, as well as the precise form of the class $[\mc{E}^*_X]$ inserted in the SCORE formula.
The main theorem is stated as follows:
\begin{mr}[Theorem \ref{thm:CI}]
Let $V$ be a smooth projective toric variety of dimension $n$. Let $Y_k, k=1,...,m (m\leq n-3)$ be smooth hypersurfaces in $V$ defined by hyperplane setions $f_k\in H^0(V,\mc{O}(H_k))$, such that every $X_k:=Y_1\cap ...\cap Y_k$ is smooth. Let the vector bundles $\mc{E}^*_{X_k}$ be small deformations of the cotangent bundle $\Omega_{X_k}$ defined by the middle cohomology of \eqref{coh-def:E-CI} such that $H^1(\mc{E}^*_{X_k})\cong W$. Then for $\sigma_i\in H^1(\mc{E}^*_{X_m}), i=1,...,n-m$, their cohomology product can be computed on the ambient space $V$ via the following SCORE formula, with $\gamma_k\in W$ determined by the presentation of $\mc{E}^*_{X_k}$ as in \eqref{cond:well-def-complex}.
\begin{equation}
\langle \sigma_1,\sigma_2,...,\sigma_{n-m}\rangle_X = 
\langle \sigma_1,\sigma_2,...,\sigma_{n-m}, \gamma_1,...,\gamma_{m}\rangle_V.
\end{equation}
\end{mr}

The contents of the rest of the paper are as follows: In Section \ref{sec:hypersurface-setting}, we set up the hypersurface case in detail and state a SCORE formula in this case. In Section \ref{sec:hypersurface-proof} we prove the hypersurface case SCORE formula after introducing two double complexes than contains all the relevant information, by doing various diagram chases. In Section \ref{sec:CI} we state and prove the Main Theorem, and prove it using induction. We set up the assumptions of the theorem in a way so that the proof of the induction step is completely analogous to the hypersurface case, with the only extra work (proof of Proposition \ref{prop-CI}) being the check of the vanishing of some sheaf cohomology groups. We conclude with a remark explaining that our assumption on $H^1(\mc{E}^*_{X_k})$ in the Main Theorem is not a strong constraint.

\textbf{Acknowledgements} The author would like to thank Qizheng Yin and BiMRC for the hospitality and helpful  discussions. The author also thank Ron Donagi and Ilarion Melnikov for helpful comments and clarifications. This paper is based on work supported in part by the National Natural Science Foundation of China (No. 12101044) and Beijing Institute of Technology Research Fund for Young Scholars.

\section{The setup of the hypersurface case}\label{sec:hypersurface-setting}
We work over the complex numbers $\mathbb{C}$. Let $V$ be a smooth projective toric variety of dimension $n$ . The cotangent bundle $\Omega_V$ fits in a SES called the toric Euler sequence (\cite{cox2011toric}, Chapter 8):
\begin{equation}\label{ses:toric-euler}
0 \to \Omega_V \to  \oplus\mc{O}(-D_i) \xrightarrow{E_0} \mc{O}\otimes W \to 0,
\end{equation}
where $W\cong \mathbb{C}^r$ is a complex vector space, $r$ is the rank of Pic($V$), and the direct sum is taken over all toric invariant divisors $D_i$. The map
$E_0$ can be explicitly written down using Cox coordinates of $V$. We call a vector bundle $\mc{E}^*_V$  \textit{toric deformed bundle}, if it is the kernel of a map $E$, which is a deformation of $E_0$. In other words, $\mc{E}^*_V$ fits in the SES
\begin{equation}\label{ses:toric-euler-deformed}
0 \to \mc{E}_V^* \to  \oplus\mc{O}(-D_i) \xrightarrow{E} \mc{O}\otimes W \to 0.
\end{equation}
Let $\sigma_i \in H^1(\mc{E}^*_V), i =1,...,n$, denote the image of $(\sigma_1,...,\sigma_n)$ under the cohomology multiplication map $$c_V: H^1(\mc{E}^*_V) \times ... \times H^1(\mc{E}^*_V) \to H^n(\wedge^n\mc{E}^*_V)$$ by $\langle \sigma_1,...,\sigma_n \rangle_V$.
In \cite{donagi2014mathematical}, the authors computed the cohomology ring $H^*_\mc{E}(V):=\bigoplus H^q(\wedge^p\mc{E}^*_V)$. They defined a Stanley-Reisner ideal $SR(V,\mc{E})$ and showed that under the identification $H^1(\mc{E}^*_V)\cong W$, one has $H^*_\mc{E}(V)\cong \Sym^*W/SR(V,\mc{E})$. In the process, they used the Koszul resolution of \eqref{ses:toric-euler-deformed}, in the form
\begin{equation}\label{les:E-Koszul}
\begin{array}{ll}
0 &\to \wedge^n\mc{E}_V^* \to  \wedge^n (\oplus\mc{O}(-D_i)) \to \wedge^{n-1}(\oplus\mc{O}(-D_i))\otimes W \\
&\to ... \to (\oplus\mc{O}(-D_i))\otimes \sym[n-1]W \to \mc{O}\otimes  \sym[n]W \to 0
\end{array}
\end{equation}
to show that the induced map $\sym[n] W \to H^n(\wedge^n\mc{E}^*_V)$ is compatible with the multiplication $c_V$ .

Now 
let $X$ be a smooth hypersurface in $V$ defined by a hyperplane section $f$. The relationship between the cotangent bundle $\Omega_V$ of $V$ and the cotangent bundle $\Omega_X$ of $X$ is encoded in the SES
\begin{equation}\label{ses:def-Omega_X}
0 \to \mc{O}_X(-X) \xrightarrow{J_0} \Omega_V|_X \to \Omega_X \to 0,
\end{equation}
where $J_0=(\partial f)$, if one view $\Omega_V$ as a subbundle of $\oplus \mc{O}(-D_i)$. In this paper we consider vector bundles $\mc{E}^*_X$ that are deformations of $\Omega_X$ which are determined by the SES
\begin{equation}\label{ses:def-E}
0 \to \mc{O}_X(-X) \xrightarrow{J}  \mc{E}_V^*|_X \to \mc{E}^*_X \to 0.
\end{equation}
Equivalently, $\mc{E}^*_X$ is a vector bundle that is the middle cohomology of the complex\footnote{
The interested reader could consult \cite{mcorist2009summing}, Section 4 for the physics significance of this complex.
}
\begin{equation}\label{coh-def:E}
 \mc{O}_X(-X) \xrightarrow{J}  \oplus\mc{O}(-D_i) \xrightarrow{E} \mc{O}\otimes W.
\end{equation}


Let $f$ be the defining equation of $X\subset V$. \eqref{coh-def:E} being a complex indicates $E\circ J$ is zero on $X$. So it is of the form $\gamma\cdot f$ for some $\gamma \in W$. 

Next, recall the Poincare residue map produces a SES
\begin{equation}\label{ses:poincare}
0 \to \Omega^n_V \to  \Omega^n_V(X) \xrightarrow{P.R.}  \Omega^{n-1}_X \to 0,
\end{equation}
where the isomorphism $\Omega_V^{n}|_X \cong \Omega^{n-1}_X \otimes \mc{O}_X(-X)$ induced from \eqref{ses:def-Omega_X} is used to write $\Omega^n_V(X)|_X$ as $\Omega^{n-1}_X$. Analogously, we have
\begin{equation}
\wedge^n\mc{E}_V^*(X)|_X \cong \wedge^{n-1}\mc{E}_X^*,
\end{equation}
and a SES
\begin{equation}\label{ses:wedge-n-e-star}
0 \to \wedge^n\mc{E}_V^* \to  \wedge^n\mc{E}_V^*(X) \to  \wedge^{n-1}\mc{E}_X^* \to 0.
\end{equation}
This induces a map $\delta: H^{n-1}( \wedge^{n-1}\mc{E}_X^*)\to H^n(\wedge^n\mc{E}_V^*)$.
\begin{prop}
When $\mc{E}^*_V$ and $\mc{E}^*_X$ are small deformations of $\Omega_V$ and $\Omega_X$ respectively, $\delta: H^{n-1}( \wedge^{n-1}\mc{E}_X^*)\to H^n(\wedge^n\mc{E}_V^*)$ is an isomorphism.
\end{prop}
\begin{proof}
This is true for the undeformed case, which is easily seen from the vanishing of related cohomologies of the middle term in \eqref{ses:poincare} by Kodaira vanishing theorem. So it is true for small deformations.
\end{proof}

Now we are ready to state the theorem of this section:
\begin{thm}[SCORE formula for hypersurfaces\label{thm:hypersuface}]
Let $X$ be a smooth hyperplane section of $V$, and  $\mc{E}^*_X$ be a small deformation of the cotangent bundle $\Omega_X$ defined by the middle cohomology of \eqref{coh-def:E}. Then under the identification $H^1(\mc{E}^*_X)\cong W$, the cohomology multiplication $H^1(\mc{E}^*_X)\times ... \times H^1(\mc{E}^*_X) \to H^{n-1}(\wedge^{n-1}\mc{E}^*_X)$ can be conducted on $V$. Namely, there exists a $\gamma=\gamma(\mc{E}^*_X)\in W$, such that $\langle \sigma_1,...\sigma_{n-1} \rangle_X = \langle \sigma_1,...,\sigma_{n-1}, \gamma \rangle_V$.
\end{thm}

We remark that $\gamma(\mc{E}^*_X)$ will be shown to be exactly the $\gamma$ such that $E\circ J =\gamma\cdot f$, as explained below \eqref{coh-def:E}. The proof of the theorem will be given in next section, mainly consists of diagram chases on two double complexes, which allows us to establish the commutativity of the two squares in the following diagram respectively.
{
\newcommand{\A}{\sym[n-1]W}
\newcommand{\B}{H^{n-1}(\wedge^{n-1}\mc{E}_X^*)}
\newcommand{\nA}{H^0(S_1(X)|_X)}
\newcommand{\nnA}{\sym[n]W}
\newcommand{\nnB}{H^{n}(\wedge^{n}\mc{E}_V^*)}
%
%
\begin{equation}
\xymatrix{
~\A		\ar[r]	\ar[d]		&		\nA	\ar[r] \ar[d]	&	~\nnA		\ar[d]
\\
~\B		\ar@{=}[r]			&		\B	\ar[r]^{\delta}	& \nnB
.}
\end{equation}
}

\section{Two double complexes and consequences}\label{sec:hypersurface-proof}
In this section we build two double complexes and use them to give the proof of Theorem \ref{thm:hypersuface}.
\subsection{Two double complexes}
In view of \eqref{ses:toric-euler-deformed}, we write the Koszul resolution of $\wedge^n\mc{E}^*_V$ as 
$0\to \wedge^n\mc{E}^*_V \to Z_\bullet$, with $Z_i = \wedge^{i}(\oplus \mc{O}(-D_i))\otimes \sym[n-i]W$. Combine it with the SESs 
\begin{equation}\label{ses:zzz}
0 \to Z_i \to Z_i(X) \to Z_i(X)|_X\to 0,
\end{equation}
we get the following double complex with exact rows and columns:

{
\newcommand{\seszb}{0}
\newcommand{\seszc}{0}
\newcommand{\seszd}{0}

\newcommand{\sesaa}{0}
\newcommand{\sesab}{\wedge^n\mc{E}_V^*}
\newcommand{\sesac}{\wedge^n\mc{E}_V^*(X)}
\newcommand{\sesad}{\wedge^{n-1}\mc{E}_X^*}
\newcommand{\sesae}{0}
\newcommand{\sesba}{0}
\newcommand{\sesbb}{\wedge^n (\oplus\mc{O}(-D_i))}
\newcommand{\sesbc}{\wedge^n (\oplus\mc{O}(-D_i))(X)}
\newcommand{\sesbd}{\wedge^n (\oplus\mc{O}_X(-D_i))(X)}
\newcommand{\sesbe}{0}
\newcommand{\sesca}{0}
\newcommand{\sescb}{(\oplus\mc{O}(-D_i))\otimes \sym[n-1]W}
\newcommand{\sescc}{(\oplus\mc{O}(-D_i))(X)\otimes \sym[n-1]W}
\newcommand{\sescd}{(\oplus\mc{O}_X(-D_i))(X)\otimes \sym[n-1]W}
\newcommand{\sesce}{0}
\newcommand{\sesda}{0}
\newcommand{\sesdb}{\mc{O}\otimes \sym[n]W}
\newcommand{\sesdc}{\mc{O}(X)\otimes \sym[n]W}
\newcommand{\sesdd}{\mc{O}_X(X)\otimes \sym[n]W}
\newcommand{\sesde}{0}
\newcommand{\sesyb}{0}
\newcommand{\sesyc}{0}
\newcommand{\sesyd}{0}
%
%
\begin{equation}\label{diag:PR}
\resizebox{350pt}{!}{\xymatrix{
& \seszb \ar[d] & \seszc \ar[d] & \seszd \ar[d] &  \\
\sesaa \ar[r] & \sesab \ar[r]\ar[d] & \sesac \ar[r]\ar[d] & \sesad \ar[r]\ar[d] & \sesae \\
\sesba \ar[r] & \sesbb \ar[r]\ar@{-->}[d] & \sesbc \ar[r]\ar@{-->}[d] & \sesbd \ar[r]\ar@{-->}[d] & \sesbe \\
\sesca \ar[r] & \sescb \ar[r]\ar[d] & \sescc \ar[r]\ar[d] & \sescd \ar[r]\ar[d] & \sesce \\
\sesda \ar[r] & \sesdb \ar[r]\ar[d] & \sesdc \ar[r]\ar[d] & \sesdd \ar[r]\ar[d] & \sesde \\
& \sesyb & \sesyc& \sesyd  & 
}}
\end{equation}
}
Each column in the double complex can be broken into SESs. For the first column, we have 
\begin{equation}\label{ses:szs}
0 \to S_i \to Z_i \to S_{i-1}\to 0,  i=1,...,n,
\end{equation}
where $S_i := {\rm Ker} (Z_i \to Z_{i-1}), i=1,...,n$ and $S_0:=\sym[n]W\otimes \mc{O}$.
This induces maps
$\delta_i: H^{i-1}(S_{i-1}) \to H^i(S_i)$
 on cohomology. This is analyzed in detail in \cite{donagi2014mathematical}. In short, it gives a surjective morphism
\begin{equation}
\delta_V: \sym[n] W \to H^n(\wedge^n\mc{E}_V^*),
\end{equation}
where $\delta_V = \delta_n \circ \delta_{n-1}\circ ...\circ \delta_1$. In view of \eqref{ses:zzz}, the second and third columns are broken into 
\begin{equation}\label{ses:szsx}
0 \to S_i(X) \to Z_i(X) \to S_{i-1}(X)\to 0,  i=1,...,n,
\end{equation}
and
\begin{equation}\label{ses:szs-restricted}
0 \to S_i(X)|_X \to Z_i(X)|_X \to S_{i-1}(X)|_X\to 0,  i=1,...,n,
\end{equation}
respectively.

Now we introduce the second double complex.
In view of the $J$ map in \eqref{coh-def:E}, the definition of $\mc{E}^*_X$, define $\mc{F}$ to be ${\rm Coker} J$. We then have a diagram of exact rows and columns: 
{
\newcommand{\seszb}{0}
\newcommand{\seszc}{0}
\newcommand{\seszd}{0}
\newcommand{\sesaa}{0}
\newcommand{\sesab}{\mc{O}_X(-X)}
\newcommand{\sesac}{\mc{O}_X(-X)}
\newcommand{\sesba}{0}
\newcommand{\sesbb}{\mc{E}_V^*|_X}
\newcommand{\sesbc}{\oplus\mc{O}_X(-D_i)}
\newcommand{\sesbd}{\mc{O}_X\otimes W}
\newcommand{\sesbe}{0}
\newcommand{\sesca}{0}
\newcommand{\sescb}{\mc{E}_X^*}
\newcommand{\sescc}{\mc{F}}
\newcommand{\sescd}{\mc{O}_X\otimes W}
\newcommand{\sesce}{0}
\newcommand{\sesyb}{0}
\newcommand{\sesyc}{0}
\newcommand{\sesyd}{0}
%
%
%
%
\begin{equation}\label{diag:def-of-E^*}
\resizebox{250pt}{!}{\xymatrix{
& \seszb \ar[d] & \seszc \ar[d] &  &  \\
& \sesab \ar@{=}[r]\ar[d] & \sesac \ar[d] & &  \\
\sesba \ar[r] & \sesbb \ar[r]\ar[d] & \sesbc \ar[r]\ar[d] & \sesbd \ar[r]\ar@{=}[d] & \sesbe \\
\sesca \ar[r] & \sescb \ar[r]\ar[d] & \sescc \ar[r]\ar[d] & \sescd \ar[r] & \sesce \\
& \sesyb & \sesyc  &  & 
}}
\end{equation}
}
Observe from the last row that $\mc{F}$ is locally free (as an extension of locally free modules). 
Take the Koszul resolution of $\wedge^n\mc{E}^*_X$ using the last row, and combine it with the last column of \eqref{diag:PR}, we have 
{
\newcommand{\seszb}{0}
\newcommand{\seszc}{0}
\newcommand{\seszd}{0}

\newcommand{\sesaa}{0}
\newcommand{\sesab}{\wedge^{n-1}\mc{E}_X^*}
\newcommand{\sesac}{\wedge^{n-1}\mc{E}_X^*}
\newcommand{\sesae}{0}
\newcommand{\sesba}{0}
\newcommand{\sesbb}{\wedge^{n-1}\mc{F}}
\newcommand{\sesbc}{\wedge^{n}(\oplus\mc{O}_X(-D_i))(X)}
\newcommand{\sesbd}{\wedge^{n}\mc{F}(X)}
\newcommand{\sesbe}{0}
\newcommand{\sesca}{0}
\newcommand{\sescb}{\mc{F}\otimes \sym[n-2]W}
\newcommand{\sescc}{\wedge^2(\oplus\mc{O}_X(-D_i))(X)\otimes \sym[n-2]W}
\newcommand{\sescd}{\wedge^2\mc{F}(X)\otimes \sym[n-2]W}
\newcommand{\sesce}{0}
\newcommand{\sesda}{0}
\newcommand{\sesdb}{\mc{O}_X\otimes \sym[n-1]W}
\newcommand{\sesdc}{\oplus\mc{O}_X(-D_i)(X)\otimes \sym[n-1]W}
\newcommand{\sesdd}{\mc{F}(X)\otimes \sym[n-1]W}
\newcommand{\sesde}{0}
\newcommand{\seseb}{0}
\newcommand{\sesec}{\mc{O}_X(X)\otimes \sym[n]W}
\newcommand{\sesed}{\mc{O}_X(X)\otimes \sym[n]W}
\newcommand{\sesyb}{0}
\newcommand{\sesyc}{0}
\newcommand{\sesyd}{0}
%
%
\begin{equation}\label{diag:exterior-induced}
\resizebox{350pt}{!}{\xymatrix{
& \seszb \ar[d] & \seszc \ar[d] & &  \\
& \sesab \ar@{=}[r]\ar[d] & \sesac \ar[d] & \seszd \ar[d] &  \\
\sesba \ar[r] & \sesbb \ar[r]\ar@{-->}[d] & \sesbc \ar[r]\ar@{-->}[d] & \sesbd \ar[r]\ar@{-->}[d] & \sesbe \\
\sesca \ar[r] & \sescb \ar[r]\ar[d] & \sescc \ar[r]\ar[d] & \sescd \ar[r]\ar[d] & \sesce \\
\sesda \ar[r] & \sesdb \ar[r]\ar[d] & \sesdc \ar[r]\ar[d] & \sesdd \ar[r]\ar[d] & \sesde \\
& \seseb & \sesec \ar@{=}[r]\ar[d] & \sesed \ar[d] & \\
&& \sesyc& \sesyd   & 
}}
\end{equation}
}
To see the exactness of the rows, recall that $\mc{F}$ fits in the SES
\begin{equation}
0\to \mc{O}_X(-X) \to \oplus\mc{O}_X(-D_i)\to \mc{F}\to 0.
\end{equation}
Applying the exterior power functor $\wedge^k$, $k=1,2,3,...$, and note that $\mc{O}_X(-X)$ is a line bundle, we get the exactness inductively.
\subsection{Proof of Theorem \ref{thm:hypersuface}}

Similarly to the argument below the SESs \eqref{ses:szs}, we break down the first two columns of \eqref{diag:exterior-induced} into SESs, the second of which is exactly the third column of \eqref{diag:PR}. Hence we adopt the notions in \eqref{ses:szs-restricted}. They induced natural maps on cohomologies, which fits into the commutative diagram 
{
\newcommand{\A}{\sym[n-1]W}
\newcommand{\B}{H^{n-1}(\wedge^{n-1}\mc{E}_X^*)}
\newcommand{\nA}{H^0(S_1(X)|_X)}
\newcommand{\nnA}{\sym[n]W}
\newcommand{\nnB}{H^{n}(\wedge^{n}\mc{E}_V^*)}
%
%
\begin{equation}\label{diag:hypersurface-cd-1}
\xymatrix{
~\A		\ar[r]	\ar[d]		&		\nA	\ar[d]	
\\
~\B		\ar@{=}[r]			&		\B		
.}
\end{equation}
}
Next, using \eqref{ses:szs} and \eqref{ses:szs-restricted} from the first and third columns of \eqref{diag:PR}, we have the following commutative diagram in cohomologies:
{%
\newcommand{\A}{\sym[n-1]W}
\newcommand{\B}{H^{n-1}(\wedge^{n-1}\mc{E}_X^*)}
\newcommand{\nA}{H^0(S_1(X)|_X)}
\newcommand{\nnA}{H^1(S_1)}
\newcommand{\nnB}{H^{n}(\wedge^{n}\mc{E}_V^*)}
%
%
\begin{equation}\label{diag:hypersurface-cd-2}
\xymatrix{
		\nA	\ar[r] \ar[d]	&	~\nnA		\ar[d]_{}
\\
		\B	\ar[r]^{\delta}	& \nnB
.}
\end{equation}
}
Next, we show that 
\begin{prop}\label{prop-hypersurface-alpha-map}
There exists a map $\alpha: H^0(S_1(X)|_X) \to \sym[n]W$ such that the following diagram commutes.
{
\newcommand{\A}{\sym[n-1]W}
\newcommand{\B}{H^{n-1}(\wedge^{n-1}\mc{E}_X^*)}
\newcommand{\nA}{H^0(S_1(X)|_X)}
\newcommand{\nnA}{\sym[n]W}
\newcommand{\nnB}{H^{n}(\wedge^{n}\mc{E}_V^*)}
%
%
\begin{equation}\label{diag:hypersurface-cd-3}
\xymatrix{
\nA	\ar@{-->}[r]^{\alpha} \ar[dr]	&	~\nnA		\ar[d]
\\
& ~ H^1(S_1)
.}
\end{equation}
}
\end{prop}
\begin{proof}
From the double complex \eqref{diag:PR}, we get the following diagram with exact rows and columns:
{%
\newcommand{\seszb}{0}
\newcommand{\seszc}{0}
\newcommand{\seszd}{0}
\newcommand{\sesaa}{0}
\newcommand{\sesab}{S_1}
\newcommand{\sesac}{S_1(X)}
\newcommand{\sesad}{S_1(X)|_X}
\newcommand{\sesae}{0}
\newcommand{\sesba}{0}
\newcommand{\sesbb}{Z_1}
\newcommand{\sesbc}{\oplus\mc{O}(-D_i)(X)\otimes\sym[n-1]W}
\newcommand{\sesbd}{\oplus\mc{O}_X(-D_i)(X)\otimes\sym[n-1]W}
\newcommand{\sesbe}{0}
\newcommand{\sesca}{0}
\newcommand{\sescb}{S_0}
\newcommand{\sescc}{\mc{O}(X)\otimes\sym[n]W}
\newcommand{\sescd}{S_0(X)|_X}
\newcommand{\sesce}{0}
\newcommand{\sesyb}{0}
\newcommand{\sesyc}{0}
\newcommand{\sesyd}{0}
%
%
%
%
\begin{equation}\label{diag:s1s0}
\tiny\xymatrix{
& \seszb \ar[d] & \seszc \ar[d] & \seszd \ar[d] &  \\
\sesaa \ar[r] & \sesab \ar[r]\ar[d] & \sesac \ar[r]\ar[d] & \sesad \ar[r]\ar[d] & \sesae \\
\sesba \ar[r] & \sesbb \ar[r]\ar[d] & \sesbc \ar[r]\ar[d]^{E} & \sesbd \ar[r]\ar[d] & \sesbe \\
\sesca \ar[r] & \sescb \ar[r]\ar[d] & \sescc \ar[r]\ar[d] & \sescd \ar[r]\ar[d] & \sesce \\
& \sesyb & \sesyc  & \sesyd  & 
}
\end{equation}
Note that $H^i(Z_1) = H^i(\oplus\mc{O}(-D_\rho))\otimes \sym[n-1]W=0$ from direct toric computation \cite{donagi2014mathematical}. So the map $$\sesbc \to \sesbd$$ in the middle row induces an isomorphism on cohomologies. Hence a standard diagram chasing yields a map $\alpha: H^0(S_1(X)|_X) \to H^0(S_0)$  (in the directions down, left, down, left). The maps are naturally induced, so the diagram commutes.
}
\end{proof}

Now combining the diagrams \eqref{diag:hypersurface-cd-1}\eqref{diag:hypersurface-cd-2}\eqref{diag:hypersurface-cd-3}, we get a map $\tilde{\alpha}$ that makes the following diagram commutes:
{
\newcommand{\A}{\sym[n-1]W}
\newcommand{\B}{H^{n-1}(\wedge^{n-1}\mc{E}_X^*)}
\newcommand{\nA}{H^0(S_1(X)|_X)}
\newcommand{\nnA}{\sym[n]W}
\newcommand{\nnB}{H^{n}(\wedge^{n}\mc{E}_V^*)}
%
%
\begin{equation}
\xymatrix{
~\A	\ar[r]^{\tilde{\alpha}} \ar[d]	&	~\nnA		\ar[d]
\\
\B	\ar[r]^{\delta}	& \nnB
.}
\end{equation}
}
Moreover, tracking along the diagrams \eqref{diag:s1s0}\eqref{diag:exterior-induced}, one finds that the map $\tilde{\alpha}$ fits in
$$
\resizebox{350pt}{!}{\xymatrix{
~H^0(\mc{O}_X)\otimes\sym[n-1]W\ar[r]^{J\ \ \ \ \ \ \ }\ar[rrd]^{\tilde{\alpha}} & H^0(\oplus\mc{O}(-D_i)(X))\otimes\sym[n-1]W\ar[r]^{\ \ \ \ \ \ \ E} & H^0(\mc{O}(X))\otimes\sym[n]W\\
& & H^0(\mc{O})\otimes \sym[n]W\ar[u]^{\cdot f}
}}.
$$
with the fact that $E\circ J = \gamma \cdot f$, we conclude that $\tilde{\alpha}$ is conducted by multiplying $\gamma$. And this finishes the proof of Theorem \ref{thm:hypersuface}.

\section{SCORE: the complete intersection case}\label{sec:CI}

We now consider the case when $X$ is a complete intersection in the toric variety $V$. In particular, we set up $X$ as successive intersections of hypersurfaces in $V$, so that the SCORE formula can be derived inductively.

Let $Y_k, k=1,...,m$ be smooth hypersurfaces in $V$ defined by $f_k\in H^0(V,\mc{O}(H_k))$, such that $X_k:=Y_1\cap ...\cap Y_k$ is smooth. When $m=1$, $X_1=X$ is a hypersurface in $V$, define the bundle $\mc{E}^*_{X_1}$ and $\mc{F}_1$ as $\mc{E}^*_X$ and $\mc{F}$ in Section \ref{sec:hypersurface-setting}. In short, they fit in the following diagram with exact rows and columns:
{
\newcommand{\seszb}{0}
\newcommand{\seszc}{0}
\newcommand{\seszd}{0}
\newcommand{\sesaa}{0}
\newcommand{\sesab}{\mc{O}_{X_1}(-H_1)}
\newcommand{\sesac}{\mc{O}_{X_1}(-H_1)}
\newcommand{\sesba}{0}
\newcommand{\sesbb}{\mc{E}_V^*|_{X_1}}
\newcommand{\sesbc}{\oplus\mc{O}_{X_1}(-D_i)}
\newcommand{\sesbd}{\mc{O}_{X_1}\otimes W}
\newcommand{\sesbe}{0}
\newcommand{\sesca}{0}
\newcommand{\sescb}{\mc{E}_{X_1}^*}
\newcommand{\sescc}{\mc{F}_1}
\newcommand{\sescd}{\mc{O}_{X_1}\otimes W}
\newcommand{\sesce}{0}
\newcommand{\sesyb}{0}
\newcommand{\sesyc}{0}
\newcommand{\sesyd}{0}
%
%
%
%
\begin{equation}\label{diag:F_1}
\resizebox{300pt}{!}{\xymatrix{
& \seszb \ar[d] & \seszc \ar[d]^{} &  &  \\
& \sesab \ar@{=}[r]\ar[d] & \sesac \ar[d]^{J_1} & &  \\
\sesba \ar[r] & \sesbb \ar[r]\ar[d] & \sesbc \ar[r]^{E}\ar[d] & \sesbd \ar[r]\ar@{=}[d] & \sesbe \\
\sesca \ar[r] & \sescb \ar[r]\ar[d] & \sescc \ar[r]^{E}\ar[d] & \sescd \ar[r] & \sesce \\
& \sesyb & \sesyc  &  & 
}}.
\end{equation}
}
For $m\geq 2$, similar to \eqref{coh-def:E}, define $\mc{E}^*_{X_m}$ to be a vector bundle that is the middle cohomology of 
\begin{equation}\label{coh-def:E-CI}
\mc{O}_{X_m}(-H_m) \xrightarrow{J_m} \mc{F}_{m-1}|_{X_m}\xrightarrow{E} \mc{O}_{X_m}\otimes W,
\end{equation}
where $J_m$ is, with an abuse of notation, a descend of some $J_m\in {\rm Hom}(\mc{O}(-H_m),\oplus \mc{O}(-D_i))$. This also implies that there is a $\gamma_m\in W$ such that
\begin{equation}\label{cond:well-def-complex}
E\circ J_m = \gamma_m\cdot f_m\in {\rm Hom}(\mc{O}(-H_m),\mc{O}_V\otimes W).
\end{equation}
Then the bundle $\mc{F}_m$ is defined to fit in the following diagram of exact rows and columns:
{
\newcommand{\seszb}{0}
\newcommand{\seszc}{0}
\newcommand{\seszd}{0}
\newcommand{\sesaa}{0}
\newcommand{\sesab}{\mc{O}_{X_{m}}(-H_m)}
\newcommand{\sesac}{\mc{O}_{X_{m}}(-H_m)}
\newcommand{\sesba}{0}
\newcommand{\sesbb}{\mc{E}_{X_{m-1}}^*|_{X_{m}}}
\newcommand{\sesbc}{\mc{F}_{m-1}|_{X_{m}}}
\newcommand{\sesbd}{\mc{O}_{X_{m}}\otimes W}
\newcommand{\sesbe}{0}
\newcommand{\sesca}{0}
\newcommand{\sescb}{\mc{E}_{X_{m}}^*}
\newcommand{\sescc}{\mc{F}_m}
\newcommand{\sescd}{\mc{O}_{X_{m}}\otimes W}
\newcommand{\sesce}{0}
\newcommand{\sesyb}{0}
\newcommand{\sesyc}{0}
\newcommand{\sesyd}{0}
%
%
%
%
\begin{equation}
\resizebox{300pt}{!}{\xymatrix{
& \seszb \ar[d] & \seszc \ar[d] &  &  \\
& \sesab \ar@{=}[r]\ar[d] & \sesac \ar[d]^{J_m} & &  \\
\sesba \ar[r] & \sesbb \ar[r]\ar[d] & \sesbc \ar[r]\ar[d] & \sesbd \ar[r]\ar@{=}[d] & \sesbe \\
\sesca \ar[r] & \sescb \ar[r]\ar[d] & \sescc \ar[r]\ar[d] & \sescd \ar[r] & \sesce \\
& \sesyb & \sesyc  &  & 
}}.
\end{equation}
}

With this setup, we can show that the argument for the hypersurface can be carried over, resulting a SCORE formula for the toric complete intersection case. 
\begin{thm}[SCORE formula for toric complete intersections]\label{thm:CI}
Let $V$ be a smooth projective toric variety of dimension $n$. Let $Y_k, k=1,...,m (m\leq n-3)$ be smooth hypersurfaces in $V$ defined by hyperplane setions $f_k\in H^0(V,\mc{O}(H_k))$, such that every $X_k:=Y_1\cap ...\cap Y_k$ is smooth. Let the vector bundles $\mc{E}^*_{X_k}$ be small deformations of the cotangent bundle $\Omega_{X_k}$ defined by the middle cohomology of \eqref{coh-def:E-CI} such that $H^1(\mc{E}^*_{X_k})\cong W$. Then for $\sigma_i\in H^1(\mc{E}^*_{X_m}), i=1,...,n-m$, their cohomology product can be computed on the ambient space $V$ via the following formula, with $\gamma_k\in W$ determined by the presentation of $\mc{E}^*_{X_k}$ as in \eqref{cond:well-def-complex}.
\begin{equation}
\langle \sigma_1,\sigma_2,...,\sigma_{n-m}\rangle_X = 
\langle \sigma_1,\sigma_2,...,\sigma_{n-m}, \gamma_1,...,\gamma_{m}\rangle_V.
\end{equation}
\end{thm}

\begin{proof} We prove the theorem by induction on $m$. When $m=1$, this is the hypersurface case Theorem \ref{thm:hypersuface}. Assume the theorem holds for $m-1$, we prove the case for $m$ by showing that $\langle \sigma_1,\sigma_2,...,\sigma_{n-m}\rangle_{X_{m}} = 
\langle \sigma_1,\sigma_2,...,\sigma_{n-m}, \gamma_m\rangle_{X_{m-1}}.$ The proof is completely analogous to the one given in Section \ref{sec:hypersurface-proof}, with the following proposition in place of Proposition \ref{prop-hypersurface-alpha-map}.
\end{proof}
\begin{prop}\label{prop-CI}
Let $\mc{K}$ be the kernel of the map $$ {\mc{F}_{m-1}(H_m)|_{X_{m}}\otimes \sym[n-m]W} \to {\mc{O}_{X_{m}}(H_m)\otimes \sym[n-m+1]W}.$$Then there exists a map $\alpha: H^0(\mc{K}) \to \sym[n-m+1]W$ makes the following diagram commute:
{
\newcommand{\A}{\sym[n-1]W}
\newcommand{\B}{H^{n-m}(\wedge^{n-m}\mc{E}_{X_{n-m}}^*)}
\newcommand{\nA}{H^0(S_1(H_m)|_{X_m})}
\newcommand{\nnA}{\sym[n-m+1]W}
\newcommand{\nnB}{H^{n-m+1}(\wedge^{n-m+1}\mc{E}_{X_{n-m+1}}^*)}
%
%
\begin{equation}\label{diag:CD-CI-2}
\xymatrix{
%
H^0(\mc{K})	\ar@{-->}[r]^{\alpha} \ar[d]&	~\nnA		\ar[d]
\\
\B	\ar[r]^{\delta\ \ \ } &  \nnB
.}
\end{equation}
}
\end{prop}
\begin{proof}
We first introduce notations similar to the hypersurface case so that the analogy between this proposition and Propsition \ref{prop-hypersurface-alpha-map} is apparent. 
The map defining $\mc{K}$ appears in the bottom right of the following diagram analogous to \eqref{diag:PR} in Section \ref{sec:hypersurface-proof}:
{
\newcommand{\seszb}{0}
\newcommand{\seszc}{0}
\newcommand{\seszd}{0}

\newcommand{\sesaa}{0}
\newcommand{\sesab}{\wedge^{n-m+1}\mc{E}_{X_{m-1}}^*}
\newcommand{\sesac}{\wedge^{n-m+1}\mc{E}_{X_{m-1}}^*(H_m)}
\newcommand{\sesad}{\wedge^{n-m}\mc{E}_{X_{m}}^*}
\newcommand{\sesae}{0}
\newcommand{\sesba}{0}
\newcommand{\sesbb}{\wedge^{n-m+1} \mc{F}_{m-1}}
\newcommand{\sesbc}{\wedge^{n-m+1} \mc{F}_{m-1}(H_m)}
\newcommand{\sesbd}{\wedge^{n-m+1} \mc{F}_{m-1}(H_m)|_{X_{m}}}
\newcommand{\sesbe}{0}
\newcommand{\sesca}{0}
\newcommand{\sescb}{\mc{F}_{m-1}\otimes \sym[n-m]W}
\newcommand{\sescc}{\mc{F}_{m-1}(H_m)\otimes \sym[n-m]W}
\newcommand{\sescd}{\mc{F}_{m-1}(H_m)|_{X_{m}}\otimes \sym[n-m]W}
\newcommand{\sesce}{0}
\newcommand{\sesda}{0}
\newcommand{\sesdb}{\mc{O}_{X_{m-1}}\otimes \sym[n-m+1]W}
\newcommand{\sesdc}{\mc{O}_{X_{m-1}}(H_m)\otimes \sym[n-m+1]W}
\newcommand{\sesdd}{\mc{O}_{X_{m}}(H_m)\otimes \sym[n-m+1]W}
\newcommand{\sesde}{0}
\newcommand{\sesyb}{0}
\newcommand{\sesyc}{0}
\newcommand{\sesyd}{0}
\newcommand{\seskb}{\mathbb{S}_1}
\newcommand{\seskc}{\mathbb{S}_1(H_m)}
\newcommand{\seskd}{\mathbb{S}_1(H_m)|_{X_m}}
\begin{equation}\label{diag:PR-CI}
\resizebox{350pt}{!}{\xymatrix{
& \seszb \ar[d] & \seszc \ar[d]_{} & \seszd \ar[d]^{} &  \\
\sesaa \ar[r] & \sesab \ar[r]\ar[d] & \sesac \ar[r]\ar[d] & \sesad \ar[r]\ar[d] & \sesae \\
\sesba \ar[r] & \sesbb \ar[r]\ar@{-->}[d] & \sesbc \ar[r]\ar@{-->}[d] & \sesbd \ar[r]\ar@{-->}[d] & \sesbe \\
\sesca \ar[r] & \sescb \ar[r]\ar[d] & \sescc \ar[r]\ar[d] & \sescd \ar[r]\ar[d] & \sesce \\
\sesda \ar[r] & \sesdb \ar[r]\ar[d] & \sesdc \ar[r]\ar[d] & \sesdd \ar[r]\ar[d] & \sesde \\
& \sesyb & \sesyc& \sesyd  & 
}}
\end{equation}
From it we get 
\begin{equation}
\resizebox{350pt}{!}{\xymatrix{
& \seszb \ar[d] & \seszc \ar[d]_{} & \seszd \ar[d] &  \\
\sesaa \ar[r] & \seskb \ar[r]\ar[d] & \seskc \ar[r]\ar[d] & \seskd \ar[r]\ar[d] & \sesae \\
\sesca \ar[r] & \sescb \ar[r]\ar[d] & \sescc \ar[r]\ar[d] & \sescd \ar[r]\ar[d] & \sesce \\
\sesda \ar[r] & \sesdb \ar[r]\ar[d] & \sesdc \ar[r]\ar[d] & \sesdd \ar[r]\ar[d] & \sesde \\
& \sesyb & \sesyc& \sesyd  & 
}}.
\end{equation}
So $\mc{K} = \seskd $ and it suffices to show that $\alpha$ makes the diagram 
\begin{equation}
\xymatrix{
~ H^0(\mathbb{S}_1(H_m)|_{X_m})\ar@{-->}[r]^{\alpha}\ar[dr] & ~\sym[n-m+1]W\ar[d]\\
&H^1(\mathbb{S}_1) 
}
\end{equation}
commute.
Analogous to Proposition \ref{prop-hypersurface-alpha-map}, it suffices to show that 
$$\sescc \to \sescd $$ induces an isomorphism on $H^0$. This is true if $H^0(\mc{F}_{m-1}) = H^1(\mc{F}_{m-1}) =0$.  Tracking from the last row of \eqref{diag:F_1}, We have
\begin{equation}\label{les:E-F}
0 \to H^0 (\mc{E}^*_{X_{m-1}}) \to H^0(\mc{F}_{m-1}) \to W \xrightarrow{\delta_0} H^1(\mc{E}^*_{X_{m-1}}) \to H^1(\mc{F}_{m-1}) \to 0.
\end{equation}
By our assumption, the connecting morphism $\delta_0$ is an isomorphism. Hence $H^1(\mc{F}_{m-1})=0$. Note that for the undeformed case, $\mc{E}^*_{X_{m-1}}$ is the cotangent bundle, and $H^0(\Omega_{X_{m-1}}) = 0$. By the semi-continuity of cohomology dimensions, $H^0(\mc{E}^*_{X_{m-1}}) = 0$. Hence $H^0(\mc{F}_{m-1})=0$. The rest of the proof is identical to that of Proposition \ref{prop-hypersurface-alpha-map}.
}
\end{proof}

Finally, we remark that $H^1(\mc{E}^*_{X_k} )\cong W$ is not a strong constraint. By the Lefschetz hyperplane theorem, it is always true when $\mc{E}^*_{X_k}$ is the cotangent bundle and $\dim X_k \geq 3$. Then by \eqref{les:E-F} and a semi-continuity argument one can conclude that for bundles $\mc{F}_{k}$ in a flat family $H^1(\mc{E}^*_{X_k} )\cong W$ always holds.
 When $\dim X_m =2$, $H^1(\Omega_{X_m})$ might be bigger. But examining the proof of Proposition \ref{prop-CI} one finds that the SCORE formula still holds for the ``toric part" of $H^1(\mc{E}^*_{X_m})$, namely the image of $W$ in $H^1(\mc{E}^*_{X_m})$.

\bibliography{CI.bbl}
\bibliographystyle{plain}
\end{document}